\documentclass[12pt]{amsart}
 \usepackage{amscd}
\usepackage{amsfonts}
\usepackage[dvipsnames]{xcolor}
\usepackage{amsmath,amssymb,amsthm} 
\usepackage{ulem}  
\usepackage[compatible]{algpseudocode}
\usepackage{algorithm}

\usepackage{url}

\usepackage{cite}
\usepackage[all]{xy}
 \usepackage{amscd}
\usepackage{amsfonts}
\usepackage{amsmath,amssymb,amsthm} 
\usepackage{ulem} 
\usepackage[compatible]{algpseudocode}
\usepackage{algorithm}
\usepackage{cite}
\usepackage{url}
\usepackage[all]{xy}

\newtheorem{thm}{Theorem}[section]

\theoremstyle{definition}
\newtheorem{ex}[thm]{Example}
\newtheorem{lem}[thm]{Lemma}

\theoremstyle{definition}

\numberwithin{equation}{section}

\def\H{\mathbb{H}} 
\def\P{\mathbb{P}} 
\def\N{\mathbb{N}} 
\def\Z{\mathbb{Z}} 
\def\C{\mathbb{C}} 
\def\G{\mathcal{G}} 
\def\F{\mathcal{F}^{d}}
\def\SF{Sing(\F)} 
\def\om{GRevLex}

\makeatletter
\newenvironment{breakablealgorithm}
  {
   \begin{center}
     \refstepcounter{algorithm}
     \hrule height.8pt depth0pt \kern2pt
     \renewcommand{\caption}[2][\relax]{
       {\raggedright\textbf{\ALG@name~\thealgorithm} ##2\par}%
       \ifx\relax##1\relax 
         \addcontentsline{loa}{algorithm}{\protect\numberline{\thealgorithm}##2}%
       \else 
         \addcontentsline{loa}{algorithm}{\protect\numberline{\thealgorithm}##1}%
       \fi
       \kern2pt\hrule\kern2pt
     }
  }{
     \kern2pt\hrule\relax
   \end{center}
  }
\makeatother

\begin{document}     


\title[Euler-Betti Algorithm]
{The Euler-Betti Algorithm to identify foliations in the Hilbert scheme} 

\author[Pantale\'on-Mondrag\'on]{P.~Rub\'i Pantale\'on-Mondrag\'on}
\address{P.~Rub\'i Pantale\'on-Mondrag\'on\\
         Centro de Investigaci\'on en Matem\'aticas, A.C.\\
         Jalisco S/N, Col. Valenciana\\
         36023, Guanajuato, Gto.  \\
         M\'exico}
\email{pantaleon.rubi@gmail.com}

\author[Mart\'in del Campo]{Abraham Mart\'in del Campo}
\address{Abraham Mart\'in del Campo\\
         Centro de Investigaci\'on en Matem\'aticas, A.C.\\
         Jalisco S/N, Col. Valenciana\\
         36023, Guanajuato, Gto.  \\
         M\'exico}
\email{abraham.mc@cimat.mx}
\urladdr{http://personal.cimat.mx:8181/~{}abraham.mc/}


\begin{abstract}
  Foliations in the complex projective plane are uniquely determined
   by their singular locus, which is in correspondence with a
   zero-dimensional ideal. Ho\-wever, this correspondence is not
   surjective.  We give conditions to determine whether an ideal
   arises as the singular locus of a foliation or not. Furthermore, we
   give an effective method to construct the foliation in the positive
   case.
\end{abstract}

\maketitle

\section{Introduction}\label{S:Intro} 

In a broad sense, a foliation is a decomposition of a manifold into
disjoint equidimensional connected submanifolds, called leaves.
Classifying foliations is still a widely open problem due the
difficulty of finding examples with specific properties. For instance,
foliations without invariant curves are known to play an important
role in the minimal exceptional problem \cite{CLN88}. However,
finding this type of foliations is difficult, even though they form a
dense open set in the space of foliations, as demonstrated by
Jouanolou \cite{J06}.

In the complex projective plane $\P^{2}$, one-dimensional foliations
of degree, at least two with isolated singularities, are uniquely
determined by their singular locus, which is a finite set of points.
This was first proved by G\'omez-Mont and Kempf \cite{GMK89} when the
singular locus consists of different points; and, later by Campillo and
Olivares \cite{CO99, CO01} in full generality. Hence, a foliation of
degree $d$ with isolated singularities is in correspondence with a
vanishing ideal of $d^2{+}d{+}1$ points.  These results translate the
study of these foliations to analyzing their singular locus, which is
an algebraic scheme inside the Hilbert scheme.  Nevertheless, it is
not known the way the space of foliations sets inside the Hilbert
scheme.

We give conditions to determine whether a point in the Hilbert scheme
can be realized as the singular locus of a foliation. We derive an
effective algorithm that constructs the foliation from a
zero-dimensional ideal in order to facilitate potential future
applications or classification of foliations.

Our algorithm relies on standard tools from computational algebraic
geo\-metry, such as the computation of syzygies via Gr\"obner bases;
thus, its complexity is not polynomial on the number of
indeterminates. However, in low dimensions our algorithm provides a
computational tool to find new and interesting examples of foliations.
For instance, an implementation of our algorithm was used by the first
author, to find a new family of foliations of degree three without
invariant curves and a unique singular point \cite{AP19}.  For
illustration, let us consider the ideal
\[
  \langle y+a z^2 + b z^5+ c z^8, z^{13}\rangle, \quad \mbox{ for
  }a,b,c \in \C^{*}.
\]
This ideal is an element of the Hilbert scheme of $13$ points.  If the
coefficients satisfy the quadratic equation $b^2{-}ac=0$; then, the
ideal corresponds to the foliation given by the $1$-form
$Adx + Bdy + Cdz $, with
\begin{align*}
  A &=(-2b/c)xy^{3}+(-2a^{3}/c)x^{2}yz+(-4a^{2}/b)y^{2}z^{2}+(-2a^{4}/c)xz^{3},\\
  B &=(2b/c)x^{2}y^{2}+(1/a)y^{3}z+(3ab/c)xyz^{2}+(a^{3}/b)z^{4}, \\
  C &=(2a^{3}/c)x^{3}y+(-1/a)y^{4}+(ab/c)xy^{2}z+(2a^{4}/c)x^{2}z^{2}+(-a^{3}/b)yz^{3}.
\end{align*}
This is a foliation of degree three, with singular locus consisting of
only one point, as we will see in the Section~\ref{Result}.

The structure of this article is the following. In
Section~\ref{S:background} we give a brief summary of the basic
concepts for which this work was developed. In
Section~\ref{Euler-Betti} we develop a characterization of the
singular scheme of a foliation and the main algorithm. In Section
~\ref{Result}, we apply the algorithm in some examples.

\section{Background}\label{S:background}


A {\bf foliation} $X^d$ on $\P^{2}$ of dimension one and degree
$d\geq 2$ is defined, up to a scalar factor, by a reduced homogeneous
$1$-form \[A(x,y,z)dx+B(x,y,z)dy+C(x,y,z)dz,\] where $A, B$, and
$C\in \C[x,y,z]$ are homogeneous polynomials of degree $d+1$, that
satisfy Euler's condition:
\[
  xA + y B+z C=0.
\]

The vanishing points of the ideal generated by the polynomials $A,B,C$
are the {\bf singular points} of the foliation, so the singular
locus of $X^d$ is defined as the algebraic variety corresponding to
the zero locus of the ideal $J=\langle A,B,C\rangle$, this is
\small{\[ {\bf Sing(X^d)} := \mathcal V(J) = \{ (x:y:z) \in \P^2
    \mid A(x,y,z) = B(x,y,z) =C(x,y,z) = 0\}.
\]}

Foliations and their singular locus can also be defined in terms of
vector fields, as summarized in \cite{A18,GMOB89}. When the
polynomials $A, B, C$ have no common factors, the singular locus
consists of a finite set of points. In general, the dimension and
degree, as well as other important information of $Sing(X^d)$ can be
retrieved from the ideal $J$ through its Hilbert polynomial and its
Betti numbers, whose definitions we will now recall.

Regard $R=\C[x,y,z]$ as a graded ring
$R=\bigoplus_{s\in \N} R_{s}$, and consider the $R$-module $M=R/J$.  In
this setting, a graded free resolution for $M$ is an exact sequence

\begin{equation*}\label{eq:freeResolution}
 \footnotesize{ \SelectTips{cm}{}
  M_{\bullet}: \xymatrix @C1.3pc @R0.9pc {
    0\ar[r]&\bigoplus\limits_{j}
    R(-j)^{\beta_{3,j}}\ar[r]^-{\varphi_{3}}&\bigoplus\limits_{j}
    R(-j)^{\beta_{2,j}}\ar[r]^{\varphi_{2}}&\bigoplus\limits_{j}
    R(-j)^{\beta_{1,j}}\ar[r]^{\varphi_{1}}&\bigoplus\limits_{j}
    R(-j)^{\beta_{0,j}}\ar[r]^-{\varphi_{0}}&M },}
\end{equation*}
where each $\bigoplus R(-j)^{\beta_{i,j}}$ is a graded free module,
each $\varphi_{i}$ is a degree $0$ homomorphism between them.  The
kernel of $\varphi_{0}$ is isomorphic to $J$, and the cokernel of
$\varphi_{1}$ is isomorphic to $M$. The exponents $\beta_{i,j}$ are
called {\bf Betti numbers} and represent the minimal number of
generators of degree $j$ for the $i$th-module
$\bigoplus R(-j)^{\beta_{i,j}}$. The maximal index $i$ where
$\beta_{ij}\neq 0$ is the {\bf length} of the resolution.  Since
$R$ is a polynomial ring in three variables, for any ideal
$J\subseteq R$, the length of the resolution $M_{\bullet}$
is at most $3$, due the Hilbert's syzygy theorem.

When $R/J=\bigoplus_{s\in \N} (R/J)_{s}$ is regarded as a commutative
$\C$-algebra, we can consider its Hilbert function, which is defined
for every $s\in \Z$ as the dimension of the $\C$-vector space
$(R/J)_{s}$.  When $s$ is sufficiently large, the Hilbert function
coincides with a polynomial called the {\bf Hilbert polynomial}.

For an ideal $J$ whose variety consists of $N$ points counted with
multiplicities, its Hilbert polynomial is the constant polynomial $N$.
The set of all saturated homogeneous zero-dimensional ideals in
$R=\C[x,y,z]$ with Hilbert polynomial equals to $N$ is the
{\bf Hilbert scheme of $N$ points in $\P^{2}$}, and we denote it
by $\H^{N}(\P^{2})$. This Hilbert scheme is a smooth, irreducible,
algebraic variety of dimension $2N$ (see~\cite{MS04, ACG11}).

Notice that a finite set of points in $\P^{2}$ can be studied locally,
up to a change of coordinates, through the affine chart
$U_{x}:=\{(1:y:z)\in\P^{2}\}\cong \C^{2}$.  In this local setting, we
use $\H^{N}(\C^{2})$ to denote the Hilbert scheme of $N$ points in
$\C^{2}$, that parameterizes ideals in $\C[y,z]$ whose Hilbert
polynomial is the constant $N$. Furthermore, the foliation $X^d$ can
be defined locally by a $1$-form
\[
  f(y,z)dy+g(y,z)dz,
\] 
for certain polynomials $f,g \in \C[y,z]$ related to the polynomials
$A,B$, and $C$.  With this local representation, we can see that a
point $(1:y:z)\in U_{x}$ lie on $Sing(X^d)$ if and only if $(y,z)$ lie
on the variety $\mathcal V(f,g)$.  Thus, the singular points of the
foliation can be studied by analyzing the variety $\mathcal V(f,g)$.
One of the advantages of this local representation is that the number
of singular points can be computed by the Milnor number which in a
broad sense, counts the multiplicity of the intersection between the
curves defined by $f$ and $g$ at a given point.  Nevertheless, for a
foliation $X^d$ of degree $d$ with isolated singularities, the total
number of singular points equals to $d^{2}+d+1$, as shown by
Jouanolou~\cite{B15}.

Since $\SF$ is finite, the $\C$-vector space
$\C[x,y,z]/\langle A,B,C\rangle$ is also finite, when $A, B, C$ have
no common factors. The degree of $\SF$ is then
$\dim_{\C} \C[x,y,z]/\langle A,B,C\rangle=d^{2}+d+1$.  This is also
true for the variety $\mathcal V(f,g)$, and after a change of
coordinates, we can assume that the equality
$\dim_{\C} \C[y,z]/\langle f,g\rangle=d^{2}+d+1$ holds.  As a
consequence, the Hilbert polynomial of both ideals is the constant
polynomial $d^{2}+d+1$; thus, we can consider the ideal
$J=\langle A,B,C\rangle$ as a point of the Hilbert scheme
$\H^{d^{2}+d+1}(\P^{2})$, or equivalently,
$\langle f,g\rangle \in\H^{d^{2}+d+1}(\C^{2})$.

\section{Results}\label{Euler-Betti}

\subsection{Criterion for foliations}

A foliation of degree $d\geq 2$ with isolated singularities is
uniquely determined by its singular subscheme~\cite{GMK89,CO01}.
Moreover, Campillo and Olivares~\cite{CO01} showed that, for a
foliation $X^{d}$ of degree $d\geq 2$, the ideal defining the singular
scheme $\SF$ must contain (up to scalar multiples) three unique
polynomials $A,B,C$ of degree $d{+}1$ satisfying Euler's
condition. Thus, $X^{d}$ is the only foliation having the algebraic
scheme $\mathcal V(A,B,C)$ as its singular locus.  Recasting these
conditions, we develop an algorithm to determine whether a given ideal
$J$ defines the singular scheme $\SF$ of a foliation, by constructing
the triplet of polynomials $A,B,C$ prescribed by Campillo y Olivares.

We begin by deriving conditions on the elements of a minimal Gr\"obner
basis of the ideal defining the singular scheme of a foliation.  So,
consider an ideal $I\subseteq \C[y,z]$, with
$I\in \H^{d^{2}+d+1}(\C^{2})$ and $d\geq2$, and let $J$ be the
homogenization of $I$ in $\C[x,y,z]$.

\begin{lem}\label{Clau-Rub}
  Let $\G\subseteq\C[y,z]$ be a minimal Gr\"obner basis of the ideal
  $I$ with respect to a graded monomial order $\succ$. If the basis
  $\G$ contains three linearly independent polynomials of degree
  $d{+}1$, then, so does $J$.  Moreover, if $J$ only contains
  polynomials of degree at least $d{+}1$, then, the converse holds.
\end{lem}

\begin{proof} 
  The first part follows from the fact that the elements of a minimal
  Gr\"obner basis are linearly independent on $\C$, and their
  homogenization preserves the linear independence.

  On the other hand, for the converse, we assume that $J$ contains at
  least $3$ linearly independent polynomials of degree $d{+}1$; and,
  no polynomials of degree $d$ or less.  Then, the set $\G$ also
  contains no polynomials of degree $d$ or less, as the homogenization
  preserves degrees, and there would be elements in $J$ of degree at
  most $d$ otherwise.
    
  The set $\G$ must contain at least one element of degree $d{+}1$.
  Otherwise, a homogeneous polynomial $F(x,y,z)\in J$ of degree
  $d{+}1$ would dehomogenize as a polynomial $F(1,y,z)\in I$ of degree
  $d{+}1$, at most. If $F(1,y,z)$ had degree less than $d{+}1$; then,
  it could not be generated by the Gr\"obner basis $\G$, because
  $\succ$ is a graded order.
  
  Suppose $\G=\{f_{1},g_{2},\ldots,g_{s}\}$ where $deg(f_{1})=d+1$ and
  $deg(g_{i})\geq d+2$ for all $i=2,\ldots, s$. Let
  $F_{3}(x,y,z)\in J$ be a polynomials of degree $d+1$ which is
  linearly independent with the homogenization $F_{1}\in J$ of
  $f_{1}$. Let's denote by $f_{3}:= F_{3}(1,y,z)$ then
  $deg(f_{3})=d+1$; otherwise, $f_{3}\in I$ could not be generated by
  $\G$. Moreover, since $\succ$ is a graded order, its leading term
  $in_{\succ}(f_{3})=c_{1}\cdot in_{\succ}(f_{1})$ for some
  $c_{1}\in \C^{*}$. Hence, the $S$-polynomial
  $S(f_{1},f_{3}):=f_{1}-\frac{1}{c_{1}}f_{3}\in I$ is not zero
  because $F_{1}$ and $F_{3}$ are linearly independent and then so are
  $f_{1}$ and $f_{3}$.  Moreover, $deg(S(f_{1},f_{3}))\leq d+1$ and
  $in_{\succ}(S(f_{1},f_{3}))\prec in_{\succ}(f_{1})$. Thus,
  $deg(S(f_{1},f_{3}))=d+1$ and
  $in_{\succ}(S(f_{1},f_{3}))\in in_{\succ}(\G-\{f_{1}\})$, which is a
  contradiction, because $\G$ is minima.

  Then suppose $\G=\{f_{1},f_{2},g_{3},\ldots, g_{s}\}$ where
  $deg(f_{1})=deg(f_{2})=d+1$ and $deg(g_{i})\geq d+2$ for all
  $i=3,\ldots, s$. Similar to the previous case, there is
  $F_{3}(x,y,z)\in J$ a lineal independent polynomial of degree $d+1$
  with the homogenization $F_{1}$ and $F_{2}$ of $f_{1}$ and $f_{2}$
  respectively, and such that $deg(F_{3}(1,y,z))=d+1$. Let
  $f_{3}:= F_{3}(1,y,z)\in I-\G$. We assume that the leading term
  $in_{\succ}(f_{3})=c_{1} \cdot in_{\succ}(f_{1})$ for some
  $c_{1}\in\C^{*}$.  We consider the $S$-polynomial
  $S(f_{1},f_{3}):= f_{1}-\frac{1}{c_{1}}f_{3}$, which is a polynomial
  not zero because $f_{1}$ and $f_{3}$ are linearly independent, then
  $deg(S(f_{1},f_{3}))\leq d+1$ and
  $in_{\succ}(f_{1})\succ in_{\succ}(S(f_{1},f_{3}))$. If
  $deg(S(f_{1},f_{3}))=d+1$, as
  $in_{\succ}(f_{1})\succ in_{\succ}(S(f_{1},f_{3}))$ and
  $deg(in_{\succ}(g_{i}))\geq d+2$, then
  $in_{\succ}(S(f_{1},f_{3}))=c_{2}\cdot in_{\succ}(f_{2})$ for some
  $c_{2}\in\C^{*}$.

  Again, we consider the $S$-polynomial
  $S_{1}:=S(f_{2},S(f_{1},f_{3}))=f_{2}-\frac{1}{c_{2}}S(f_{1},f_{3})=f_{2}-
  \frac{1}{c_{2}}f_{1}+\frac{1}{c_{2}c_{1}}f_{3}\in I$. Since
  $f_{1}, f_{2}$ and $f_{3}$ are linearly independent then
  $S_{1}\neq 0$, thus, $in_{\succ}(h)\succ in_{\succ}(S_{1})$ for all
  $h\in\G$, that is, $in_{\succ}(S_{1})$ is not an element in the ideal
  of leading terms of $I$, but this is a contradiction because
  $S_{1}\in I$.
  \end{proof}

  Lemma~\ref{Clau-Rub} gives the first conditions to determine whether
  an ideal in the Hilbert Scheme corresponds to a foliation. However,
  these conditions are not enough as we can see in the following
  example.

\begin{ex}\label{EjemploTeroema}
  We consider the ideal given by
  $I=\langle y^{2}z+z^{3}, yz-y^{3}-z^{2}y\rangle \subset \C[y,z]$.  The
  generators of this ideal are the components of the $1$-form
  $yzdz+(y^{2}+z^{2})(zdy-ydz)$ which represent a foliation on
  $\P^{2}$ of degree $d=2$ \cite{CDGM10}.

  If we consider the minimal Gr\"obner basis
  \[\G_{1}=\{yz^{2}, y^{2}z+z^{3}, y^{3}-yz,z^{4}\}\] with respect to
  the graded reverse lexicographical monomial order $\om$.  The
  homogenization of the polynomials of degree $3$ with respect to $x$
  is the set \[\G=\{yz^{2}, y^{2}z+z^{3}, y^{3}-xyz\}\] and we can
  show that $xf_{1}+yf_{2}+zf_{3}\neq 0$ for all $f_{j}\neq f_{i}$,
  and $f_{i},f_{j}\in \G$.

  But with the homogenizations of the elements of degree $3$ of the
  minimal basis Gr\"obner
  \[\G_{2}=\{yz^{2}, y^{2}z+z^{3}, -y^{3}-xyz-z^{2}y\}\] with respect
  to the same monomial order  we
  have \[x(yz^{2})+y(y^{2}z+z^{3})+z(-y^{3}-xyz-z^{2}y)=0.\]
\end{ex}

As we saw in the example, the problem here is that there is no
uniqueness in the minimal bases. The previous theorem only guarantees
the existence of polynomials of the correct degree that would define
the foliation; however, these could not satisfy Euler's condition.

\subsection{ The Euler's Condition}

Let $J\in \H^{d^{2}+d+1}(\P^{2})$ be an ideal generated by three
homogeneous polynomials
$F_{1}(x,y,z),F_{2}(x,y,z),F_{3}(x,y,z)\in\C[x,y,z]$ of degree $d+1$,
which are linearly independent and without common factors. Moreover,
we assume that this set of generators is minimal in the degrees, that
is, there are no polynomials of degree $\leq d$ in $J$.

Denote again by $R=\C[x,y,z]$. The ideal
$J=\langle F_{1},F_{2},F_{3}\rangle$ as an $R$-module has a minimal
free resolution of length $1$ (see~\cite[Proposition 3.1]{E05}).
Since finitely generated graded modules have only one minimal graded
resolution up to isomorphism, one can derive a resolution for $M=R/J$
from a resolution for $J$, and by the Hilbert-Burch theorem
(see~\cite[Theorem3.2]{E05}), one deduces that the minimal free
resolution of $M$ is represented in the form:

\begin{equation}\label{resolucionM}
  \SelectTips{cm}{}
  M_{\bullet}: \xymatrix @C1.7pc @R0.9pc {
    0\ar[r]& R({-}b_{1})\oplus R({-}b_{2})\ar[r]^-{\varphi_{2}}\ar[d]&
    R({-}d{-}1)^{3}\ar[r]^-{\varphi_{1}}\ar[d]& R\ar[r]^-{\varphi_{0}}&M,\\
    &Syz(J)\ar[ru]& J\ar[ur] }
\end{equation}
where $\varphi_{0}$ is the quotient map,
$\varphi_{1}=\begin{pmatrix} F_{1} & F_{2} & F_{3}\end{pmatrix}$, and
$\varphi_{2}=(m_{ij})$ is a $3{\times }2$-matrix whose non-zero
entries are homogeneous polynomials of degree $b_{j}{-}d{-}1$, and the
syzygy module $Syz(J)$, which is the kernel of $\varphi_{1}$
consisting of all algebraic relations between the polynomials
$F_{1},F_{2},F_{3}$. Moreover, we can assume that
$b_{1}\geq b_{2}\geq d+1$ (\cite[Proposition 1.9]{E05}). Thus, we
have the following result.

\begin{thm}\label{Rubi}
  There is a column $\Lambda$ in the presentation $\varphi_{2}$ on the
  resolution (\ref{resolucionM}) whose entries generate a
  $2$-dimensional or $3$-dimensional $\C$-vector space $W$ in $R_{1}$
  if, and only if $b_{2}=d+2$. In particular, the polynomials
  $F_{1},F_{2},F_{3}$ satisfy the Euler's condition if, and only if,
  the vector space $W$ is $3$-dimensional.
\end{thm}

\begin{proof}
  We suppose that $W=\langle \ell_{1},\ell_{2},\ell_{3}\rangle_{\C}$
  is a $2$-dimensional or $3$-dimensional $\C$-vector subspace of
  $R_{1}$. Let $\ell_{i}$ be a generator of $W$. Since that
  $M_{\bullet}$ is a zero-degree resolution then, without loss of
  generality, $deg(\ell_{i})=b_{2}-d-1$. Since $W\subset R_{1}$ then
  $deg(\ell_{i})=1$, thus, $b_{2}=d+2$.  Conversely, if $b_{2}=d+2$;
  then, there exists a column $\Lambda$ in the presentation
  $\varphi_{2}$ with linear entries. Let $W$ be the space generate for
  these entries. If the $\dim_{\C} W=1$, then
  $\Lambda=\begin{pmatrix} \lambda_{1}m & \lambda_{2}m &
    \lambda_{3}m\end{pmatrix}^{\top}$, where $m$ is the generator of
  $W$ and $\lambda_{i}\in\C$ for $i=1,2,3$.  Since that
  $\varphi_{1}\cdot\varphi_{2}=0$, then
  $\sum_{i=1}^{3}\lambda_{i}F_{i}=0$. Since, $F_{1},F_{2}$ and $F_{3}$
  are linearly independent polynomials, thus $\lambda_{i}=0$ for all
  $i$ and this is a contradiction. Then $\dim_{\C}W=2$ or $3$.

  For the second part, we suppose that $xF_{1}+yF_{2}+zF_{3}=0$, then
  $\begin{pmatrix}x & y & z\end{pmatrix}^{\top}\in Syz(J)$. We can
  consider the matrix $\begin{pmatrix}x & y & z\end{pmatrix}^{\top}$
  like a column of the presentation $\varphi_{2}$ and the
  $\C$-vectorial space $W$ generate by these entries is
  $3$-dimensional. On the other hand, if $dim_{\C}(W)=3$, we can see
  this space as a $R$-module, and we can define the $R$-homomorphism
  $\phi$ between the module $W$ and the module $R$ that sends
  generators to generators.
\end{proof}

Condition $b_{2}=d+2$ of the previous Theorem is equivalent to having
that the Betti number $\beta_{2,d+2}\neq 0$.

\subsection{Description and proof of correctness of the Euler-Betti
  Algorithm}

With the results above, we present the algorithm \ref{algoritmo2}. It
allows us to verify the existence of three polynomials of the correct
degree that satisfy Euler's condition in a given ideal; and,
consequently we obtain a foliation associated with this ideal.  In the
proof, we assume that there are three polynomials of the correct
degree, because this is the first condition verified in the ideal
given.

\begin{breakablealgorithm}
\caption{{\bf: The Euler-Betti Algorithm}}\label{algoritmo2}
\begin{algorithmic}[1]
  \vspace{0.2cm}
  \REQUIRE{An ideal
    $J\in\H^{d^2+d+1}(\P^{2})$.}
  \vspace{0.3cm}
  \ENSURE{Determine if there is a foliation of degree $d$ such that
    $J$ is its singular subscheme. In an affirmative case, give a
    representation of the foliation.} \vspace{0.5cm}

  \STATE{Compute the Betti numbers of the module
    $M=\C[x,y,z]/ J$.}  \vspace{0.3cm}

  \STATE{\textbf{If} the Betti numbers $\beta_{1,d+1}=3$,
    $\beta_{1,s}=0$ for all $s<d+1$ and $\beta_{2,d+2}\neq 0$,
    \textbf{then}:}\label{Al2} \vspace{0.3cm}

  \STATE{Choose an element $\ell\in Syz(J)$ with entries of
    degree $1$.}  \vspace{0.3cm}\label{Al3}

  \STATE{If the dimension of the $\C$-vector space generated by the
    entries of $\ell$ is $3$, then:}\vspace{0.3cm}

  \STATE{Calculate the matrix $D$ given by the linear transformation
    that sends $x,y,z$ to the entries of $\ell$.}\label{deflis}
  \vspace{0.3cm}

  \STATE{ Calculate the matrix
              \[ T:=\begin{bmatrix} F_{1} & F_{2} & F_{3}
                \end{bmatrix}\cdot D^{-1}\]
              where $F_{i}$ are the elements of degree $d+1$ in
              $J$.}\label{trans} \vspace{0.3cm}

            \STATE{ Define $A=T_{1}$, $B=T_{2}$ and $C=T_{3}$, where
              $T_{i}$ are the columns of the matrix
              $T$.}\vspace{0.3cm} \STATE{ Print
              $(A,B,C)$} \vspace{0.3cm}

            \STATE{ \textbf{Else } Print \textquotedblleft No
              foliation".}
\end{algorithmic}
\end{breakablealgorithm}

\begin{proof}[ Proof of correctness of the algorithm \ref{algoritmo2}]
  We assume that $\beta_{1,d+1}=3$ and $\beta_{1,s}=0$ for all
  $s<d+1$, then the Betti numbers $\beta_{2,i+2}=0$ for all
  $i\leq d-1$ (\cite[Proposition 1.9]{E05}). This means, that there exist
  three independent forms of degree $d+1$, and there are not forms of
  less degree in $J$.

  Since $\beta_{2,d+2}\neq 0$, then there exist $\Lambda\in Syz(J)$
  with entries of degree $1$. If the dimension of the space generated
  by the entries of $\Lambda$ is three, by the Theorem \ref{Rubi} and
  by the lineal transformation of the step \ref{trans} and the step
  \ref{deflis}, in the ideal $J$ there exist three forms $(A,B,C)$ of
  degree $d+1$ satisfying Euler's condition.
\end{proof}

\section{Conclusions }\label{Result}
%
\subsection{Applications Euler-Betti Algorithm}

Algorithm \ref{algoritmo2} is a tool that allows us to obtain
foliation quickly from an ideal in $\C[x,y,z]$ (or equivalently
$\C[y,z]$). However, when the generators of the ideal depend on other
parameters, such as their coefficients, this problem is more
complicated to compute computationally because of the non-polynomial
complexity of the algorithm. For small degrees, it is possible to
determine the conditions of these generators as presented in the
following examples and whose calculations were carried out in the
software \cite{M2}.

\begin{itemize}
\item Let $I=\langle y^{2}-z^{3},yz^{2}\rangle \in \H^{7}(\C^{2})$. By the
  Algorithm \ref{algoritmo2} we have that $I$ corresponds to a
  foliation of degree $2$ defined by the $1$-form $Adx+Bdy+Cdz$ where
\begin{align*}
  A=& -y^{3}\\
  B=& xy^{2}-z^{3}\\
  C=& yz^{2}.
\end{align*}

\item The ideal
  $I=\langle y^{3}+14yz^{2}-8z^{3}+8y^{2}-z-14,
  -9y^{3}+11yz^{2}+3z^{3}-3y^{2}+9z-11,
  7y^{3}+yz^{2}-5z^{3}+5y^{2}-7z-1)\rangle \in\H^{7}(\C^{2})$ correspond to
  a foliation $Adx+Bdy+Cdz$ of degree $2$ too where
\begin{align*}
  A=& -y^{3}+x^{2}z\\
  B=& xy^{2}-z^{3}\\
  C=& -x^{3}+yz^{2}.
\end{align*}

\item The ideal
  $I=\langle -6y^{2}z^{3}+y^{4}+2y^{3}+3y^{2},
  4z^{4}+5z^{2}-6y^{3}z^{2}\rangle \in\H^{21}(\C^{2})$ correspond to a
  degree $4$ foliation $Adx+Bdy+Cdz$ where
  \begin{align*}
    A=&y^{5}+2xy^{4}+3x^{2}y^{3}-4z^{5}-5x^{2}z^{3}\\
    B=&6y^{2}z^{3}-xy^{4}-2x^{2}y^{3}-3x^{2}y^{2}\\
    C=&4xz^{4}+5x^{3}z^{2}-6y^{3}z^{2}.
    \end{align*}

\end{itemize}

In $2017$, Alc\'antara in \cite{A18} gave a family of foliations of
degree $d$ with a single singular point when $d$ is odd, and the local
representation of this family of foliations is given by the ideal of
the form
\begin{equation}\label{FA}
  I=\langle
  y+az^{\frac{d+1}{2}}+bz^{\frac{d^{2}+d+2}{2}}+cz^{d^{2}+\frac{d+3}{2}},
  z^{d^{2}+d+1}\rangle \in\H^{d^{2}+d+1}(\C^{2})
\end{equation}

whose coefficients of the polynomial
$y+az^{\frac{d+1}{2}}+bz^{\frac{d^{2}+d+2}{2}}+cz^{d^{2}+\frac{d+3}{2}}$
satisfy the equation $db^{2}-ac=0$ and $abc\neq 0$.

Later, in $2019$, Alc\'antara and the first author in \cite{AP19}
give a new example of a family of foliations for degree $3$, and whose
local representation has the form
\begin{equation}\label{FP}
   I=\langle y+az^{2}+bz^{5}+cz^{8},z^{13}\rangle \in\H^{13}(\C^{2})
\end{equation}
where the coefficients of $y+az^{2}+bz^{5}+cz^{8}$ satisfy the
equation $lb^{2}-ac=0$ with $l=-3,1$ and $abc\neq 0$.

For $d=3$, the two families of foliations (\ref{FA}) and (\ref{FP})
have an unique singular point. The singular point of de foliation
(\ref{FA}) is a nilpotent point and the singular point of the
foliation (\ref{FP}) is a saddle-node. The two families were also
obtained with Algorithm \ref{algoritmo2}, and both of them are
particular cases of an ideal of the form

 \begin{equation}\label{ideal1}
   I=\langle y+\sum_{i=2}^{N-1}a_{i}z^{i}, z^{N}\rangle \in \H^{N}(\C^{2})
\end{equation}

where $a_{i}\in \C$ and $N=d^{2}+d+1$ for some $d\geq 2$. Then, can we
obtain new foliations with local representation of the form $I$ as in
(\ref{ideal1}), whose coefficients satisfy a polynomial equation? What
about even degree?

With a symbolic experiment and a version of the Algorithm \ref{algoritmo2} adapted for the ideal (\ref{ideal1}), in addition to
foliations mentioned above, the following foliation was obtained. For
example, for $d=3$, if the coefficients of the polynomial
$ y+az^{2}+bz^{5}+cz^{11}$ satisfy the equation $3b^{3}+a^{2}c=0$ with
$abc\neq 0$, then the ideal given by
  \[ I=\langle y+az^{2}+bz^{5}+cz^{11}, z^{13}\rangle \in \H^{13}(\C^{2})\]
  corresponds to the foliation $Adx+Bdy+Cdz$ where

  \begin{align*}
    A=& \frac{-3a}{b}xy^{3}+\frac{9a^{2}b}{2c}x^{2}yz+\frac{27b^{2}}{2c}y^{2}z^{2}+\frac{9a^{3}b}{2c}xz^{3},\\
    B=& \frac{3a}{b}x^{2}y^{2}+\frac{1}{a}y^{3}z+\frac{5a^{2}}{b}xyz^{2}-\frac{6ab^{2}}{c}z^{4},\\
    C=& \frac{-9a^{2}b}{2c}x^{3}y-\frac{1}{a}y^{4}+\frac{3b^{2}}{2c}xy^{2}z-\frac{9a^{3}b}{2c}x^{2}z^{2}+\frac{6ab^{2}}{c}yz^{3}
  \end{align*}

  which singular point is a nilpotent point.


\section*{Acknowledgments}

Work of Mart\'in del Campo was supported by CONACyT under grant
A1-S-30035 and it is part of the reasearch project Cátedra-1076.
Work of Pantale\'on-Modrag\'on was supported
by Conacyt under \textquotedblleft Estancias Posdoctorales por
M\'exico'' [2020-2021].

The authors would like to thank Diego Rodriguez Guzman and
Claudia Alcántara for helpful discussions on foliations and their
equations.  Pantale\'on-Mondragon would like to thank  Rafael
Ibarra for his help with grammar and spelling checks that improved
this manuscript.

\bibliographystyle{alpha} 
\bibliography{BettiEuler-Petra}

\begin{thebibliography}{CDGBM10}

\bibitem[ACG11]{ACG11}
Enrico Arbarello, Maurizio Cornalba, and Phillip~A. Griffiths.
\newblock {\em Geometry of algebraic curves. {V}olume {II}}, volume 268 of {\em
  Grundlehren der Mathematischen Wissenschaften [Fundamental Principles of
  Mathematical Sciences]}.
\newblock Springer, Heidelberg, 2011.
\newblock With a contribution by Joseph Daniel Harris.

\bibitem[Alc18]{A18}
Claudia~R. Alc\'{a}ntara.
\newblock Foliations on $cp^2$ of degree {$d$} with a singular point with
  {M}ilnor number {$d^2+d+1$}.
\newblock {\em Rev. Mat. Complut.}, 31(1):187--199, 2018.

\bibitem[APM20]{AP19}
Claudia~R. Alc\'{a}ntara and Rub\'{\i} Pantale\'{o}n-Mondrag\'{o}n.
\newblock Foliations on with a unique singular point without invariant
  algebraic curves.
\newblock {\em Geom. Dedicata}, 207:193--200, 2020.

\bibitem[Bru15]{B15}
Marco Brunella.
\newblock {\em Birational geometry of foliations}, volume~1 of {\em IMPA
  Monographs}.
\newblock Springer, Cham, 2015.

\bibitem[CDGBM10]{CDGM10}
D.~Cerveau, J.~D\'{e}serti, D.~Garba~Belko, and R.~Meziani.
\newblock G\'{e}om\'{e}trie classique de certains feuilletages de degr\'{e}
  deux.
\newblock {\em Bull. Braz. Math. Soc. (N.S.)}, 41(2):161--198, 2010.

\bibitem[CLNS88]{CLN88}
C.~Camacho, A.~Lins~Neto, and P.~Sad.
\newblock Minimal sets of foliations on complex projective spaces.
\newblock {\em Inst. Hautes \'{E}tudes Sci. Publ. Math.}, (68):187--203 (1989),
  1988.

\bibitem[CO99]{CO99}
Antonio Campillo and Jorge Olivares.
\newblock A plane foliation of degree different from 1 is determined by its
  singular scheme.
\newblock {\em C. R. Acad. Sci. Paris S\'{e}r. I Math.}, 328(10):877--882,
  1999.

\bibitem[CO01]{CO01}
Antonio Campillo and Jorge Olivares.
\newblock Polarity with respect to a foliation and {C}ayley-{B}acharach
  theorems.
\newblock {\em J. Reine Angew. Math.}, 534:95--118, 2001.

\bibitem[Eis05]{E05}
David Eisenbud.
\newblock {\em The geometry of syzygies}, volume 229 of {\em Graduate Texts in
  Mathematics}.
\newblock Springer-Verlag, New York, 2005.
\newblock A second course in commutative algebra and algebraic geometry.

\bibitem[GMK89]{GMK89}
Xavier G\'{o}mez-Mont and George Kempf.
\newblock Stability of meromorphic vector fields in projective spaces.
\newblock {\em Comment. Math. Helv.}, 64(3):462--473, 1989.

\bibitem[GMOB89]{GMOB89}
X.~G\'{o}mez-Mont and L.~Ort\'{\i}z-Bobadilla.
\newblock {\em Sistemas din\'{a}micos holomorfos en superficies}, volume~3 of
  {\em Aportaciones Matem\'{a}ticas: Notas de Investigaci\'{o}n [Mathematical
  Contributions: Research Notes]}.
\newblock Sociedad Matem\'{a}tica Mexicana, M\'{e}xico, 1989.

\bibitem[GS]{M2}
D.R. Grayson and M.E. Stillman.
\newblock Macaulay2, a software system for research in algebraic geometry.
\newblock Available at {\tt http://www.math.uiuc.edu/Macaulay2/}.

\bibitem[Jou79]{J06}
J.~P. Jouanolou.
\newblock {\em \'{E}quations de {P}faff alg\'{e}briques}, volume 708 of {\em
  Lecture Notes in Mathematics}.
\newblock Springer, Berlin, 1979.

\bibitem[MS05]{MS04}
Ezra Miller and Bernd Sturmfels.
\newblock {\em Combinatorial commutative algebra}, volume 227 of {\em Graduate
  Texts in Mathematics}.
\newblock Springer-Verlag, New York, 2005.

\end{thebibliography}
\end{document}